\documentclass{article}
\usepackage{graphicx} 
\usepackage{listings}
\usepackage{amsthm}
\usepackage{comment}
\usepackage{amsfonts}
\newtheorem{theorem}{Theorem}[section]
\newtheorem{lemma}{Lemma}[section]

\newtheorem{proposition}{Proposition}[theorem]
\newtheorem{definition}{Definition}[theorem]
\newtheorem{corollary}{Corollary}[theorem]

\usepackage[
backend=biber,
style=alphabetic,
sorting=ynt
]{biblatex}

\title{A proof that HT is more likely to outnumber HH than vice versa in a sequence of n coin flips}
\author{Simon Segert }
\date{}

\begin{document}

\maketitle

\section{Abstract}
Consider the following probability puzzle: A fair coin is flipped $n$ times. For each $HT$ in the resulting sequence, Bob gets a point, and for each $HH$ Alice gets a point. Who is more likely to win? We provide a proof that Bob wins more often for \textit{every} $n\geq 3$. As a byproduct, we derive the asymptotic form of the difference in win probabilities,  and obtain an efficient algorithms for their calculation.

\section{Introduction}
The puzzle described in the abstract was posed in \cite{litt}. It is subtle because even though both players have the same expected score, it turns out that their win probabilities are not equal. It is possible to algorithmically compute the win probabilities and thereby rigorously determine the winner for a specific value of $n$ \cite{ekhad}. However to our knowledge a proof that Bob wins for every value of n has not been given. \textbf{We provide here a formal proof that Bob wins more often for every n}.

The proof uses generating function technology, which allows us to derive several interesting and unexpected results as a byproduct. 
To state these, let $P_n(Bob)$ denote the probability of the event that Bob has more points after $n$ flips, and similarly for Alice. Let $\Delta_n=P_n(Bob)-P_n(Alice)$ be the difference in win probabilities.  We show that \textbf{the asymptotic expansion of the differences is $\Delta_n={\frac 1 {2\sqrt{n\pi}}}+O(n^{-3/2})$} and moreover that \textbf{$\Delta_n$ can be computed in $O(n)$ time}. Note that similar results had also been obtained by \cite{ekhad}, using somewhat different techniques than here.

There are two major steps of the proof: the first consists of establishing the following explicit formula for the generating function:
\begin{equation}f(t):=\sum_{n\geq 0} \Delta_n (2t)^n={\frac 1 2}{\frac 1 {\sqrt{(1-t)(1-2t)(2t^2+t+1)}}}-{\frac 1 2}{\frac 1 {1-t}},\end{equation} and the second consists of showing that the coefficients are positive.  We actually provide two independent arguments for the positivity of the coefficients, so in a sense we provide 1.5 proofs that Bob wins more often. The first argument uses closure properties of the class of absolutely monotonic functions, and the second relates $f(t)$ to the generating function of a certain manifestly positive combinatorial sequence.

\section{Proof of Main Result}

We will now prove that $\Delta_n>0$ for $n\geq 3$. We will split up the argument into three subsections. In Section \ref{gen_func_derivation}, we prove the claimed form of the generating function. We then give two independent arguments for the positivity of the coefficients of the generating function, in Sections \ref{first_proof} and \ref{second_proof}. 

\subsection{Derivation of generating function}\label{gen_func_derivation}

The objective of this section is to prove the following theorem: 

\begin{theorem} \label{gen_func_thm} The generating function $f(t)=\sum_n \Delta_n (2t)^n$ is given by
\begin{equation}f(t)={\frac 1 2}{\frac 1 {\sqrt{(1-t)(1-2t)(2t^2+t+1)}}}-{\frac 1 2}{\frac 1 {1-t}}\end{equation}
\end{theorem}

As a first step, we will derive an expression for $\Delta_n$ in terms of a complex integral \footnote{Note that, by convention, we will assume that all complex integrals are multiplied by a factor of ${\frac 1 {2\pi i}}$, which we do not explicitly write out}. 

\begin{lemma}
\label{complex_lemma}
    For any $n>0$, we have 
    \begin{equation}\Delta_n=2^{-n}\int_C \textbf{1}^t{\frac {A(z)^{n-1}-A(1/z)^{n-1}}{1-z}}\textbf{1}dz\end{equation}
where $A(z)=\left(\begin{array}{cc} 1 & 1/z\\ 1 & z\end{array}\right)$, $\textbf{1}$ is the vector of ones, and $C$ is any simple contour in the complex plane that encloses the origin.
\end{lemma}
\begin{proof}
Define the random variable $X_n$ to be the outcome of the nth flip (i.e. H or T), and define the random variable $Y_n$ to be Alice's score minus Bob's score after n flips. Finally, consider the values

\begin{equation}p_{n,k,T}=P(Y_n=k|X_n=T)\end{equation}
with $p_{n,k,H}$ defined analagously. By the chain rule of probability:

\begin{equation}p_{n,k,T}={\frac 1 2}P(Y_{n}=k|X_{n-1}=T,X_n=T)+{\frac 1 2}P(Y_{n}=k|X_{n-1}=H,X_n=T)\end{equation}

For the first term, we know that neither player can get a point for a subsequence that starts with $T$. Thus if the last two flips are $TT$ and the cumulative score is $k$, then the cumulative score must already have been $k$ after only $n-1$ flips. So the first term is $P(Y_{n}=k|X_{n-1}=T,X_n=T)=P(Y_{n-1}=k|{n-1}=T,X_n=T)=p_{n-1,k,T}$ (since $Y_{n-1}$ is independent of $X_n$). As for the second term, if the last two flips are $HT$, then this gives 1 point for Bob, and therefore the score difference after the first $n-1$ flips must have been $k+1$. So this is $p_{n-1,k+1,H}$. Thus we obtain the recurrence 
\begin{equation}p_{n,k,T}={\frac 1 2}p_{n-1,k,T}+{\frac 1 2}p_{n-1,k+1,H}\end{equation}

Similarly, we can derive the recurrence
\begin{equation}p_{n,k,H}={\frac 1 2}p_{n-1,k,T}+{\frac 1 2}p_{n-1,k-1,H}\end{equation}

Next, consider the probability generating function $\phi_{n,T}(z)=\sum_n p_{n,k,T}z^k$, with $\phi_{n,H}$ defined analagously. 

By multiplying each of the two recurrences above by $z^k$ and summing over $k$, we obtain the matrix equation:

\begin{equation}\left(\begin{array}{c}\phi_{n,T}(z)\\ \phi_{n,H}(z)\end{array}\right)={\frac 1 2}\left(\begin{array}{cc}1 & z^{-1} \\ 1 & z \end{array}\right)\left(\begin{array}{c}\phi_{n-1,T}(z)\\ \phi_{n-1,H}(z)\end{array}\right)\end{equation}

and by induction: 

\begin{eqnarray}\left(\begin{array}{c}\phi_{n,T}(z)\\ \phi_{n,H}(z)\end{array}\right)&=&{\frac 1 {2^{n-1}}}\left(\begin{array}{cc}1 & z^{-1} \\ 1 & z \end{array}\right)^{n-1}\left(\begin{array}{c}\phi_{1,T}(z)\\ \phi_{1,H}(z)\end{array}\right)\\&=&{\frac 1 {2^{n-1}}}\left(\begin{array}{cc}1 & z^{-1} \\ 1 & z \end{array}\right)^{n-1}\left(\begin{array}{c}1\\ 1\end{array}\right)\end{eqnarray}

Letting $\phi_n(z)=\sum_n P(Y_n=k)z^k$ be the probability generating function for $Y_n$, we evidently have the relation $\phi_n(z)={\frac 1 2}\phi_{n,T}(z)+{\frac 1 2}\phi_{n,H}(z)$. Thus: 

\begin{equation}\phi_{n}(z)=2^{-n}\left(\begin{array}{cc}1 & 1\end{array}\right)^T\left(\begin{array}{cc}1 & z^{-1} \\ 1 & z \end{array}\right)^{n-1}\left(\begin{array}{c}1\\ 1\end{array}\right)\end{equation}

Now, note that $Y_n$ can only take finitely many values (since neither player can get more than n-1 points after n flips). In particular, $\phi_n(z)$ is a meromorphic function, so by the residue formula: 

\begin{equation}P(Y_n=k)=\int_C z^{-1-k}\phi(z)dz\end{equation}
for any simple closed contour $C$ which contains the origin. If we take $C$ to lie completely within the (open) unit disc $|z|<1$, then we have 
\begin{eqnarray}
P_n(\textit{Bob wins})& = & \sum_{k<0} P(Y_n=k)\\
& = & \sum_{k<0}\int_C z^{-1-k}\phi(z)dz\\
& = & \sum_{k\geq 0}\int_C z^k\phi_n(z)dz\\
& = & \int_C {\frac {\phi_n(z)}{1-z}}dz
\end{eqnarray}

where the geometric sum converges by assumption on $C$. By a symmetric argument, 

\begin{equation}P_n(\textit{Alice wins})=\int_{C'} {\frac {\phi_n(z)}{1-1/z}}{\frac {dz}{z^2}}\end{equation}

where now $C'$ lies in the \textit{complement} of the (closed) unit disc $|z|\leq 1$. In particular, we can take $C'$ to be the ``inverse" of C: $C'=\{1/z: z\in C\}$. Making the change of variable $z\mapsto 1/z$ we see that 

\begin{equation}P_n(\textit{Alice wins})=\int_{C} {\frac {\phi_n(1/z)}{1-z}}dz\end{equation}

Note that there is a negative sign that arises from the fact that the function $1/z$ reverses the orientation along the curve. Therefore 

\begin{equation}\Delta_n= \int_{C} {\frac {\phi_n(z)-\phi_n(1/z)}{1-z}}dz\end{equation}
where $C$ lies in the interior of the unit disc. Note that the integrand has a removable singularity at $z=1$ and no other singularities except for $z=0$; this implies that the above formula for $\Delta_n$ is valid for any $C$ that encloses the origin (i.e. we can drop the condition that $C$ lies in the unit disc). 
\end{proof}

\begin{proof} (of Theorem \ref{gen_func_thm})

Starting with the  complex integral formula, we see that, for $t$ in a sufficiently small neighborhood of the origin, we have

\begin{equation}f(t)=t\textbf{1}^t\int_C {\frac {(I-tA(z))^{-1}-(I-tA(1/z))^{-1}}{1-z}}\textbf{1}dz\end{equation}

Doing a bit of algebra (cf. Section \ref{integrand_code}), the integral becomes 

\begin{equation}
\label{expand_inverse_eqn}
f(t)=t^3\int_C{\frac {z^3-z^2-z+1}{(t^2z^2+(t-1)z+(t-t^2))((t^2-t)z^2+(1-t)z-t^2)}}dz\end{equation}

Interestingly, both quadratic equations in the denominator have the same discriminant $\Delta(t)=4t^4-4t^3+t^2-2t+1$. By the quadratic formula, we see the integral has four poles located at: 
\begin{equation}{\frac {-t+1 \pm \sqrt{\Delta(t)}}{2t^2}}\end{equation}

\begin{equation}{\frac {t-1 \pm \sqrt{\Delta(t)}}{2t^2-2t}}\end{equation}

Now, for small $t$, we have $\sqrt{\Delta(t)}=1-t+O(t^3)$. Thus as $t\to 0$ the pole ${\frac {-t+1 + \sqrt{\Delta(t)}}{2t^2}}$ tends to infinity, while the complementary one tends to zero. Similarly for the other pair of poles, the ''+" sign one tends to zero while the other one tends to infinity. Since $C$ is fixed, we see in particular that for sufficiently small $t$ that ${\frac {-t+1 - \sqrt{\Delta(t)}}{2t^2}}$ and ${\frac {t-1 + \sqrt{\Delta(t)}}{2t^2-2t}}$ will lie in $C$ while the other two will lie outside of it. Therefore, the expression for $f(t)$ reduces to computing the residues at these two poles. Since all four poles are simple, we can directly read off the residue at one of the poles $r_i$ as  

\begin{equation}Res_{r_i}=P(r_i)/Q'(r_i)\end{equation}

where $P$ and $Q$ are respectively the numerator and denominator of the expression in the integral and the prime denotes the derivative with respect to $z$. At this point, we can verify the claimed form of $f$ through mechanical calculation (or computer algebra system, cf. Section \ref{residue_code}).

\end{proof}
\subsection{First proof of positivity}
\label{first_proof}

Our goal is to show that $f^{(k)}(0)>0$ for each $k\geq 3$. This is similar to the notion of \textit{absolute monotonicity}, but differs in that we only care about the derivatives at a single point, and moreover we want a strict inequality. It will therefore be useful at this point to study this class of functions in more detail. 

\begin{definition}Let $g$ be a function with a convergent power series expansion in a neighborhood of $0$. We say that $g$ is \textbf{strictly absolutely monotonic at the origin} (or ``\textbf{samo}" for short) if $g^{(i)}(0)>0$ for every $i\geq 0$. More generally, we say that $g$ is \textbf{k-samo} if $g^{(i)}(0)\geq 0$ for $i<k$ and $g^{(i)}(0)>0$ for $i\geq k$ \footnote{Note that 0-samo is synonymous with samo}. 
\end{definition}

We can rephrase our goal as showing that $f$ is 3-samo. We now state a few simple but useful facts which mirror the corresponding results for absolutely monotonic functions: 
\begin{proposition}
The sum of two k-samo functions is k-samo. Moreover, a positive scalar multiple of a k-samo function is k-samo. 
\end{proposition}
\begin{proof}
Clear.
\end{proof}
\begin{proposition}
\label{add_const}
Let $f$ be k-samo. Then $f-f(0)$ is $max(k,1)$-samo. 
\end{proposition}

\begin{proof}
Clear.
\end{proof}

\begin{proposition}
\label{multiply}
    
Let $f$ be $k_1$-samo and $g$ be $k_2$-samo. Then $fg$ is $k_1+k_2$-samo. 

\end{proposition}

\begin{proof} By the iterated product rule, the $kth$ derivative of the product is a positive linear combination of terms of the form $f^{(i)}(0)g^{(k-i)}(0)$ for $i=0,\dots, k$. We consider two cases, the first being $k<k_1+k_2$.  In this case we must have either $i<k_1$ or $k-i<k_2$, and therefore one of the terms in the product $f^{(i)}(0)g^{(k-i)}(0)$ is equal to zero.

Now consider the second case $k\geq k_1+k_2$. The derivative is clearly a sum of non-negative terms, so we just need to show that at least one of them is strictly positive. Setting $i=k_1$ we have $f^{(k_1)}(0)g^{(k-k_1)}(0)$. The first factor is clearly positive by assumption on $f$, and the second term is similarly positive because $k-k_1\geq k_2$. \end{proof}

\begin{proposition}
\label{log}
    If $log f$ is k-samo, then so is $f$. 
\end{proposition}

\begin{proof} Expand the derivatives of $e^{log f}$ using Faa di Bruno's formula. \end{proof}

\begin{proposition}
\label{deriv}
    If $f$ is k-samo then $\int_0^x f(t)dt$ is $k+1$-samo.
\end{proposition}
\begin{proof}
Clear.
\end{proof}

We can use the above properties of samo functions to reduce our problem to the following simpler one:

\begin{proposition}\label{reduction}
Let $h(t)$ be defined by
\begin{equation}h(t)={\frac {t+2}{2t^2+t+1}}+{\frac 1 {1-2t}}\end{equation}
If $h(t)$ is samo, then the generating function $f(t)$ is 3-samo. 
\end{proposition}

\begin{proof}
We can see by direct computation (cf. Section \ref{log_derive_code}) that
\begin{equation}{\frac {2t^2}{1-t}}h(t)={\frac d {dt}}\log {\frac {\sqrt{1-t}}{\sqrt{(1-2t)(2t^2+t+1)}}}\end{equation}

Assume $h$ is samo. The function ${\frac {2t^2}{1-t}}$ is clearly 2-samo, so by assumption on $h$ and Proposition \ref{multiply}, we conclude that RHS is 2-samo. By Proposition \ref{deriv}, we conclude that $\log {\frac {\sqrt{1-t}}{\sqrt{(1-2t)(2t^2+t+1)}}}$ is 3-samo, and by Proposition \ref{log}, we see that ${\frac {\sqrt{1-t}}{\sqrt{(1-2t)(2t^2+t+1)}}}$ is 3-samo. By Proposition \ref{add_const}, we have ${\frac {\sqrt{1-t}}{\sqrt{(1-2t)(2t^2+t+1)}}}-1$ is 3-samo. Finally, because ${\frac 1 {1-t}}$ is samo, we conclude by Proposition \ref{multiply} that 

\begin{equation}f(t)={\frac 1 {2(1-t)}}\left({\frac {\sqrt{1-t}}{\sqrt{(1-2t)(2t^2+t+1)}}}-1\right)\end{equation}
is 3-samo as desired.

\end{proof}
At this point, the natural thing is to analyze $h$. 

\begin{proposition}
\label{h_cfts}
The coefficients of the function $h(t)$ defined in Proposition \ref{reduction} are given by
\begin{equation}h(t)[t^n]=2Re(\phi^n)+2^n\end{equation}
where $\phi={\frac {-1+\sqrt{-7}}2}$ and $n\geq 0$.
\end{proposition}
\begin{proof}
It suffices to show that ${\frac {t+2}{2t^2+t+1}}=\sum_n 2Re(\phi^n)t^n$. To wit, we have for sufficiently small $t$ that 
\begin{eqnarray}
\sum_n 2Re(\phi^n)t^n& = & 2Re({\frac {1}{1-t\phi}})\\
& = & {\frac {1}{1-t\phi}}+{\frac {1}{1-t\overline{\phi}}}\\
& = & {\frac {1-t\overline{\phi}+1-t\phi}{|1-t\phi|^2}}\\
& = & {\frac {2+t}{(1+t/2)^2+7t^2/4}}\\
& = & {\frac {2+t}{1+t+2t^2}}
\end{eqnarray}
\end{proof}

As a simple corollary:

\begin{proposition}
\label{pos_cfts}
    The function $h(t)$ is samo. 
\end{proposition}
\begin{proof}
By Proposition \ref{h_cfts}, we need to show that $2^n+2Re(\phi^n)>0$ for every n. We have 
\begin{equation}Re(\phi^n)\geq -|\phi|^n=-2^{n/2}\end{equation}
and therefore the coefficients are strictly positive provided that $n\geq 3$. The cases $n=0,1,2$ are easily verified directly.
\end{proof}

Combining Propositions \ref{reduction}, \ref{h_cfts} and \ref{pos_cfts}, we have proven the following: 

\begin{theorem}
The coefficients $f(t)[t^n]$ of the generating function are strictly positive for all $n\geq 3$. In particular, Bob is more likely to win the game than Alice for any $n\geq 3$.
\end{theorem}
\subsection{Second proof of positivity}
\label{second_proof}
We now provide an alternative argument for the postivity of $\Delta_n$, by relating it to  combinatorial sequence which is manifestly positive. This proof uses the explicit form of the generating function $f(t)$ but is otherwise completely independent from the argument given in Section \ref{first_proof}.

\begin{definition}
Let $S\subset\mathbb{Z}_{\geq 0}^2-\{(0,0)\}$ be a finite subset of the non-negative plane lattice. Assume that there is an associated positive integer $c_s\geq 1$ for each $s\in S$; we denote the collection of all by $C=\{c_s\}_{s\in S}$. A \textbf{colored lattice path} is a finite sequence $\{(s_i,x_i)\}_i$ such that $s_i\in S$ and $x_i\in \{1,\dots, c_{s_i}\}$ for each i (the path may be empty). The \textbf{endpoint} of the path is defined to be $\sum_i s_i\in\mathbb{Z}^2$.
\end{definition}

Note that the values $c_s$ have an obvious interpretation as a collection of possible colors for each type of edge.  

\begin{definition}
Let $S,C$ be as above, and let $(a,b)\in\mathbb{Z}^2$. The number of colored lattice paths with endpoint $(a,b)$ is denoted by $N_{S,C}(a,b)$.
\end{definition}

We can now state the main result of this section.

\begin{theorem}
\label{lattice_thm}
The sequence $\Delta_n$ can be expressed in terms of a count of colored lattice paths:
\begin{equation}\Delta_n={\frac {N_{S,C}(n,n)-1}{2^{n+1}}}\end{equation}
where $S=\{(6,5),(0,1),(1,1),(3,3)\}$, $c_{(3,3)}=2$, and $c_{(i,j)}=1$ for $(i,j)\in S-\{(3,3)\}$. In particular $\Delta_n>0$ for $n\geq 3$.
\end{theorem}

The proof will follow easily from two simple lemmas: 

\begin{lemma}
\label{lattice_gf}
The bivariate generating function $G_{S,C}(x,y)=\sum_{(a,b)\in\mathbb{Z}_{\geq 0}^2}N_{S,C}(a,b)x^ay^b$ is given by 
\begin{equation}G_{S,C}(x,y)={\frac 1 {1-\sum_{(i,j)\in S}c_{(i,j)}x^iy^j}}\end{equation}
\end{lemma}
\begin{proof}
By considering the last segment in a path we obtain the relation:
\begin{equation}
N_{S,C}(a,b)=\delta_{a}\delta_b+\sum_{(i,j)\in S}c_{(i,j)}N_{S,C}(a-i,b-j)
\end{equation}
where $N_{S,C}(a,b)=0$ if $min(a,b)<0$ and where the $\delta$ term corresponds to the empty path. Multiplying both sides by $x^ay^b$ and summing over $a$ and $b$ gives the result.
\end{proof}
\begin{lemma}\label{diag_gf}(Stanley) 
Let $f,g,h$ be univariate polynomials with $h(0)=0$. Consider the bivariate function $G(x,y)={\frac 1 {1-xf(xy)-yg(xy)-h(xy)}}$. Then the Diagonal $D(t):=\sum_{n\geq 0}G(x,y)[x^ny^n]t^n$ is given by 
\begin{equation}D(t)={\frac 1 {\sqrt{(1-h(t))^2-4tf(t)g(t)}}}\end{equation}
\end{lemma}

\begin{proof} 
This is Exercise 6.15 in \cite{stanley},and a solution is also provided therein.  
\end{proof}
 
We now prove Theorem \ref{lattice_thm} from the two lemmas. Taking the $S$
and $C$ defined in the statment of the Theorem, we see from Lemma \ref{lattice_gf} that 
\begin{eqnarray}
    G_{S,C}(x,y)& =& {\frac 1 {1-x^6y^5-y-xy-2x^3y^3}}\\
    & = & {\frac 1 {1-x(x^5y^5)-y(1)-(2x^3y^3+xy)}}\\
    &= & {\frac 1 {1-xf(xy)-yg(xy)-h(xy)}}\\
    \end{eqnarray}
    where
\begin{eqnarray}
    f(t)&:=&t^5\\
    g(t)& := & 1\\
    h(t)& := & 2t^3+t
\end{eqnarray}
Therefore by Lemma \ref{diag_gf} the diagonal $D_{S,C}(t):=\sum_n N_{S,C}(n,n)t^n$ satisfies 
\begin{eqnarray}
D_{S,C}(t)& = & {\frac 1 {\sqrt{(1-h(t))^2-4tf(t)g(t)}}}\\
& = & {\frac 1 {\sqrt{(1-t-2t^3)^2-4t^6}}}\\
& = &{\frac 1 {\sqrt{4t^4-4t^3+t^2-2t+1}}}\\
& = & 2f(t)+{\frac 1 {1-t}}
\end{eqnarray}
where $f(t)=\sum_n \Delta_n(2t)^n$ and we have used Theorem \ref{gen_func_thm} in the last line. Theorem \ref{lattice_thm} follows immediately.

\section{Consequences of the Proof}

\subsection{Asymptotic Analysis}
 It is easy to see that $\Delta_n\to 0$, but a natural and non-obvious follow up question is how fast is the convergence? Using the explicit form of the generating function, this can answered with a simple application of the famous Darboux formula (cf. \cite{wilf}).

Notation-wise, we define $d_n:=2^n\Delta_n+1/2$, which are the coefficients of the function $\tilde{f}(t)={\frac 1 2}{\frac 1{\sqrt{(1-t)(1-2t)(2t^2+t+1)}}}$. Observe that $\tilde{f}(t)\sqrt{1/2-t}$ can be extended to a homolomorphic function in the disc $|z|<1/2+\epsilon$. Darboux's formula then implies the following asymptotic form: 

\begin{equation}2^{-n}d_n={\frac 1 2}{n-1/2\choose n}+O(n^{-3/2})={\frac 1 {2\sqrt{\pi}}}{\frac {\Gamma(n+1/2)}{\Gamma(n+1)}}+O(n^{-3/2})\end{equation}

By Stirling's formula, we in particular obtain: 

\begin{equation}\Delta_n= {\frac 1 {2\sqrt{n\pi}}} +O(n^{-3/2})\end{equation}

The same result was also obtained in \cite{ekhad} through different means.
\subsection{Recurrence Relation}
\label{recurrence_algo}
Here we provide an algorithm to compute $\Delta_n$ in $O(n)$ time. To do so, we will use the reparametrization $d_n=2^n\Delta_n+1/2$ from above and compute $d_n$. The main observation is that the generating function $\tilde{f}(t)$ for $d_n$ is an inverse square root of a rational function. In particular, it satisfies an algebraic relation of the form $p(t)\tilde{f}(t)+q(t)\tilde{f}'(t)=0$ for some polynomials $p$ and $q$. By equating the terms to zero, we see that the coefficients of $\tilde{f}$ satisfy a linear recurrence relation. Crucially, the number of terms in the recurrence relation does not depend on $n$ (i.e. it is bounded by the degrees of $p$ and $q$). Therefore, directly iterating the recurrence allows us to compute the nth term in $O(n)$ time. 

To give the details, we obtain by direct computation: 
\begin{equation}(\log \tilde{f})'(t)={\frac {-8t^3+6t^2-t+1}{(1-x)(1-2t)(2t^2+t+1)}}\end{equation}

In particular, we get an equation:

\begin{equation}(4t^4-4t^3+t^2-2t+1)\tilde{f}'(t)+(8t^3-6t^2+t-1)\tilde{f}(t)=0\end{equation}

Setting the coefficients of the LHS to zero gives the recurrence:

\begin{equation}(-8/n+4)d_{n-4}-(-6/n+4)d_{n-3}+(-1/n+1)d_{n-2}-(-1/n+2)d_{n-1}+d_{n}=0\end{equation}

We can also plug in $d_n=2^n\Delta_n+1/2$ to obtain a recurrence directly for $\Delta_n$:

\begin{eqnarray*}
\Delta_n&=&{\frac 1 {n2^n}}+\left({\frac 1 {2n}}-{\frac 1 4}\right)\Delta_{n-4}+\left(-{\frac {3}{4n}}+{\frac 1 2}\right)\Delta_{n-3}+\left({\frac 1{4n}}-{\frac 1 4}\right)\Delta_{n-2}+\left(-{\frac 1{2n}}+1\right)\Delta_{n-1}
\end{eqnarray*}

A similar recurrence relation had been obtained by \cite{ekhad} using a different argument. 

An obvious consequence is that this algorithm not only computes $\Delta_n$ in $O(n)$ time, but can in fact generate the entire sequence $\Delta_1,\dots, \Delta_n$ in $O(n)$ time.

\section{Appendix-symbolic calculations}
\label{appendix}
We provide code to verify the omitted calculations in the main text using the Python package Sympy.

\subsection{Equation \ref{expand_inverse_eqn}}
\label{integrand_code}
\begin{lstlisting}[language=Python]
from sympy import *
z,t=symbols(`z t')
P=t**3*(z**3-z**2-z+1)
Q=(t**2*z**2+(t-1)*z+(t-t**2))*((t**2-t)*z**2+(1-t)*z-t**2)
A=Matrix([[1,1/z],[1,z]])
A2=A.subs(z,1/z)
I=Matrix([[1,0],[0,1]])
integral=t*sum((I-t*A).inv()-(I-t*A2).inv())/(1-z)
simplify(integral-P/Q)#returns 0
\end{lstlisting}

\subsection{Residue calculation (Theorem \ref{gen_func_thm})}
\label{residue_code}

\begin{lstlisting}[language=Python]
from sympy import *
z,t=symbols(`z t')
P=t**3*(z**3-z**2-z+1)
Q=(t**2*z**2+(t-1)*z+(t-t**2))*((t**2-t)*z**2+(1-t)*z-t**2)
delta=4*t**4-4*t**3+t**2-2*t+1
r1=(-t+1-sqrt(delta))/(2*t*t)
r2=(t-1+sqrt(delta))/(2*t*t-2*t)
R=P/diff(Q,z)
gf=R.subs(z,r1)+R.subs(z,r2)
gf2=(1/sqrt(delta)-1/(1-t))/2
simplify(gf-gf2) #returns 0
\end{lstlisting}
\subsection{Logarithmic derivative (Proposition \ref{reduction})}
\label{log_derive_code}
\begin{lstlisting}[language=Python]
h=(t+2)/(2*t*t+t+1)+1/(1-2*t)
Q=(1-t)/((1-2*t)*(2*t*t+t+1))
simplify(diff(log(sqrt(Q)))-2*t*t*h/(1-t))#returns 0

\end{lstlisting}
\printbibliography

\end{document}